\documentclass[11pt,letterpaper]{article}

\usepackage{amssymb,amsmath,amsthm,enumitem,csquotes,tikz-cd,tikz,mathrsfs}

\numberwithin{equation}{section}

\newtheorem{thm}{Theorem}
\newtheorem{prp}{Proposition}[section]
\newtheorem{lmm}[prp]{Lemma}
\newtheorem{crl}[prp]{Corollary}

\theoremstyle{definition}
\newtheorem{dfn}[prp]{Definition}
\newtheorem{eg}[prp]{Example}

\theoremstyle{remark}
\newtheorem{rmk}[prp]{Remark}

\def\BE#1{\begin{equation}\label{#1}}
\def\EE{\end{equation}}

\def\BEnum#1{\begin{enumerate}[label=#1,leftmargin=*,topsep=-10pt,itemsep=-3pt]}
\def\EEnum{\end{enumerate}}

\def\lra{\longrightarrow}

\def\D{\mathbb D}

\def\R{\mathbb R}
\def\N{\mathbb N}

\def\Z{\mathbb Z}

\def\gee{g}

\def\supp{\tn{supp}}

\def\supp{\textnormal{supp}}
\def\area{\textnormal{area}}

\def\vol{\textnormal{vol}}
\def\KD{\textnormal{KD}}
\def\D{\textnormal{D}}
\def\mass{\textbf{M}}
\def\flat{\textbf{F}}
\def\resultyear{1976}
\def\d{\textnormal{d}}

\title{Coarsely Holomorphic Curves and Symplectic Topology}
\author{Spencer Cattalani\thanks{Partially supported by NSF grant DMS 1901979 and by the Simons Foundation}}
\date{\today}

\begin{document}

\maketitle

\begin{abstract}
A taming symplectic structure provides an upper bound on the area of an approximately pseudoholomorphic curve in terms of its homology class. We prove that, conversely, an almost complex manifold with such an area bound admits a taming symplectic structure. This confirms a speculation by Gromov. We also characterize the cone of taming symplectic structures numerically, prove that complex 2-cycles can be approximated by coarsely holomorphic curves, and provide a lower energy bound for such curves.
\end{abstract}

\tableofcontents

\section{Introduction}
\label{Introduction_sec}
Among all almost complex manifolds, those tamed by symplectic structures have a particularly robust geometry. Often, for example in Gromov's Compactness Theorem \cite{Gr}, the role of a symplectic form $\omega$ in the proof of a geometric result is to provide an upper bound on the area of a pseudoholomorphic curve $S$ in terms of the homology pairing $\langle [\omega], [S] \rangle$. Is a taming symplectic form necessary for such a bound, or can such arguments be applied to a wider class of almost complex manifolds? This paper demonstrates that a taming symplectic form is, in fact, necessary for such a bound to hold for \textit{coarsely holomorphic curves}, as was speculated in \cite[Section IV, p143]{GrSp}. These objects have much of the global rigidity of pseudoholomorphic curves, but have enough local flexibility to exist in cases where pseudoholomorphic curves do not. This makes them particularly well-suited to probe the large-scale behavior of almost complex manifolds and, in particular, answer the question of when they admit taming symplectic structures.\\

Informally, coarsely holomorphic curves are surfaces which are close, on average, to being pseudoholomorphic. More formally, let $(X,J,g)$ be a compact almost complex Riemannian manifold. No compatibility is required between $J$ and $g$. Let $\widetilde{Gr}_2(TX)$ be the space of oriented 2-planes in $TX$. Fix a continuous nonnegative ``distance'' function $D$ on $\widetilde{Gr}_2(TX)$ that vanishes precisely on the subspace of $J$-invariant planes with the orientation induced by $J$.




\begin{dfn}\label{coarse_curve_dfn}
Let $\varepsilon > 0$. A closed oriented surface $S$ in $X$ is a \textit{coarsely $\varepsilon$-holomorphic curve} if
$$\frac{\int_S \D(TS)\vol}{\int_S \vol} < \varepsilon.$$
\end{dfn}

\begin{eg}
A closed oriented surface is a pseudoholomorphic curve if and only if it is coarsely $\varepsilon$-holomorphic for every $\varepsilon > 0$.
\end{eg}


\begin{rmk}
As $X$ is compact, different choices of metric and distance function only change the value of $\varepsilon$ in the definition above. As we let $\varepsilon$ go to zero, these choices are inconsequential. The distance function we use, called the \textit{Kähler defect}, interacts nicely with the almost complex structure. It is defined in Section \ref{Prelim_sec}.
\end{rmk}

\begin{thm}\label{Gromov_conj_thm}
Let $(X,J)$ be a closed almost complex manifold. A cohomology class $h\!\in\!H^2(X,\R)$ contains a symplectic form taming $J$ if and only if there exist $\varepsilon\!>\!0$ and $C\!>\!0$ such that
\begin{equation}\label{area_bound_eqn}
\area(S) \leq C\langle h,[S]\rangle
\end{equation}
for every coarsely $\varepsilon$-holomorphic curve $S$ in $X$.
\end{thm}

\begin{rmk}
Theorem \ref{Gromov_conj_thm} was posed by Gromov as ``questionable'' in \cite[Section IV, p143]{GrSp}. There, it is stated in terms of $(\varepsilon,\delta)$-approximately $J$-holomorphic curves: a closed oriented surface~$S$ in $X$ is an \textit{$(\varepsilon,\delta)$-approximately $J$-holomorphic curve} if
$$\area(S_\varepsilon) \geq (1 - \delta)\area(S),$$
where $S_\varepsilon \subset S$ is the region on which $\D(TS) < \varepsilon$.
A coarsely $\varepsilon$-holomorphic curve is $(\varepsilon',\delta)$-approximately $J$-holomorphic for any $\varepsilon'\!>\!0$ and $\delta\!>\!0$ such that $\varepsilon'\delta < \varepsilon$. Conversely, if $M$ is the maximum of $D$ on $\widetilde{Gr}_2(TX)$, an $(\varepsilon',\delta)$-approximately $J$-holomorphic curve is
coarsely $(\varepsilon' + M\delta)$-holomorphic. Therefore, the two notions are equivalent.
\end{rmk}

\begin{eg}
The Inoue surface, as described in \cite{Ino}, is  a closed complex surface containing no holomorphic curves. Therefore, (\ref{area_bound_eqn}) is vacuous for holomorphic curves $S$. However, the Inoue surface does not admit a symplectic structure. This shows that Theorem~\ref{Gromov_conj_thm} is false if ``coarsely $\varepsilon$-holomorphic'' is replaced with ``pseudoholomorphic''.
\end{eg}

The following corollary is a numerical characterization of the symplectic cone. It is a symplectic analogue of Kleiman's criterion \cite[Theorem~IV.2.1]{Klei} from projective geometry.

\begin{crl}\label{symp_cone_crl}
Let $(X,J)$ be a closed almost complex manifold such that $J$ is tamed by some symplectic form. A cohomology class
$h\!\in\!H^2(X,\R)$ contains a symplectic form taming $J$ if and only if there exists $\varepsilon\!>\!0$ such that $h$ is strictly positive on
every coarsely $\varepsilon$-holomorphic curve.
\end{crl}

Complex cycles, as defined in \cite{Sullivan}, are key to the proof of Theorem~\ref{Gromov_conj_thm}. They have been generally overlooked in symplectic topology, but have recently been applied in \cite{CAT}. Proposition \ref{Rel_cycle_prp} below, which is a relative version of \cite[Theorem III.2]{Sullivan}, indicates their importance.

\begin{prp}\label{Rel_cycle_prp}
Let $(X,J)$ be a closed almost complex manifold and $V$ be a linear subspace of $H^2(X,\R)$. If $V$ does not contain the class of a symplectic form taming $J$, there is a complex cycle $T \neq 0$ such that $\langle h, [T] \rangle = 0$ for all $h \in V$.
\end{prp}

In particular, if ``coarsely $\varepsilon$-holomorphic curve'' were replaced with ``complex cycle'' in the statement of Theorem \ref{Gromov_conj_thm}, it would follow immediately from Proposition \ref{Rel_cycle_prp}. The definition of complex cycle is recalled in Section~\ref{Prelim_sec}. Proposition \ref{Rel_cycle_prp} shows that a geometric condition which sufficiently constrains the behavior of complex cycles implies the existence of a taming symplectic form. Therefore, questions concerning the sharpness of geometric results in symplectic topology can be reduced to ones about the geometry of complex cycles.
Theorem~\ref{Approx_thm} below is very useful in this regard. It says that complex cycles can be approximated by coarsely holomorphic curves, which are more geometric and easier to relate to pseudoholomorphic curves.

\begin{thm}\label{Approx_thm}
Let $(X,J)$ be a closed almost complex manifold. Let $T \neq 0$ be a complex cycle on~$X$. There exist sequences $\varepsilon_i >0$, $c_i > 0$, and of $S_i \subset X$ such that $S_i$ is a coarsely $\varepsilon_i$-holomorphic curve, $\varepsilon_i$ tends to zero, and $c_iS_i$ tends to $T$ in the weak topology on currents.
\end{thm}

\begin{crl}\label{Existence_crl}
Let $(X,J)$ be a closed almost complex manifold. For every $\varepsilon\!>\!0$, there exists a coarsely $\varepsilon$-holomorphic curve in $X$.
\end{crl}

\begin{rmk}
As coarsely holomorphic curves are only required to be close to being pseudoholomorphic \textit{on average}, they have no local structure and can only probe the geometry of a space on large scales. Their relation to surfaces which are \textit{uniformly} close to being pseudoholomorphic is unclear. Such surfaces have local structure, which potentially obstructs their ability to sense the coarse aspects of a space. It is not known to what extent Theorems \ref{Gromov_conj_thm} and \ref{Approx_thm} might hold if ``coarsely holomorphic'' is replaced with ``uniformly approximately holomorphic'' in their statements, but \cite{Don} contains some results in this direction.
\end{rmk}

We also show that there is a lower bound on the area, or ``energy'', of a coarsely pseudoholomorphic curve. This generalizes the classical lower energy bound for pseudoholomorphic curves \cite[Corollary~4.2]{Zinger} and demonstrates rigidity in the geometry at large scales of coarsely holomorphic curves.

\begin{thm}\label{lower_energy_bound_thm}
Let $(X,J,g)$ be a compact almost complex Riemannian manifold. There exist $\hbar > 0$ and $\varepsilon > 0$ such that
\begin{equation}\label{lower_bound_eqn}
\area(S) > \hbar  
\end{equation}
for every coarsely $\varepsilon$-holomorphic curve $S$ in $X$. If $X$ contains no pseudoholomorphic curves, $\hbar$ can be chosen arbitrarily large.
\end{thm}

The structure of the paper is as follows. Preliminary statements concerning currents are collected in Section~\ref{Prelim_sec}. The proofs of Theorems~\ref{Gromov_conj_thm}~and~\ref{Approx_thm}, assuming Proposition \ref{key_prp}, are in Section~\ref{Main_Results_sec}. The proof of Proposition~\ref{key_prp} itself is delayed until Section~\ref{Prop_Proof_sec}. It is proved via a succession of approximation lemmas, which themselves are proved using classical topology (in the style of Whitney) and geometric measure theory. In fact, no result proved after $\resultyear$ is used in the proofs of Theorems \ref{Gromov_conj_thm} and \ref{Approx_thm}. Theorem~\ref{lower_energy_bound_thm} is proved in Section~\ref{lower_energy_section} and depends on recent results in the regularity theory of semicalibrated currents.

\subsection*{Acknowledgements}
The author is thankful to Aleksander Doan for explaining his work, to Jiahao Hu for helpful comments, to Dennis Sullivan for many interesting conversations,
and to Aleksey Zinger for his great help in improving this paper over the course of many drafts. The author is also thankful to the anonymous referee for helpful comments.

\section{Preliminaries}\label{Prelim_sec}

In this section, we recall some basics from geometric measure theory. The main reference is \cite[Chapter 4, Section 1]{GMT}. While the results in the references are stated for currents in $\R^n$, by the Nash Embedding Theorem \cite{Nash}, they also apply to currents on manifolds.\\

Let $X$ be a smooth $n$-manifold. For $k \in \Z^{\geq 0}$, we denote by $\Omega^k_c(X)$ the space of compactly supported differential $k$-forms on $X$ equipped with the strong $C^\infty$-topology. A sequence of forms converges in this topology if and only if they are eventually supported in the same compact set and they converge in $C^\infty$. A \textit{k-current} is a continuous linear functional on $\Omega^k_c(X)$. The exterior derivative on forms induces a boundary operator on currents via the equation $$\partial T(\alpha) := T(d\alpha).$$
A \textit{k-cycle} is a $k$-current $T$ on $X$ such that $\partial T = 0$.
The \textit{support} of a $k$-current $T$ on $X$, denoted $\supp(T)$, is the complement of the largest open subset $U \subset X$ such that $T(\omega) = 0$ for all $\omega\in \Omega^{k}_{c}(U)$.
A sequence of $k$-currents $T_i$ on $X$ \textit{converges to a current $T$ in the weak topology on currents} if
$$\lim_{i\to\infty}T_i(\omega) = T(\omega)$$
for all $\omega \in \Omega^k_c(X).$

\begin{dfn}\label{complex_cycle_dfn}
Let $(X,J)$ be a closed almost complex manifold and $\Omega^+(X) \subset \Omega^2_c(X)$ be the space of $2$-forms $\omega$ such that $\omega(v, Jv) > 0$ for all $v \neq 0 \in TX$. A $2$-current $T$ on $X$ is a \textit{complex cycle} if $\partial T = 0$ and $T(\omega) > 0$ for all $\omega \in \Omega^+(X)$.
\end{dfn}

\begin{eg}\label{cur_int_eg}
An oriented $C^1$-submanifold $S$ of dimension $k$ induces a $k$-current, also denoted $S$, via integration. The boundary of the current induced by $S$ is the current induced by the boundary of $S$ as a submanifold. A closed embedded surface in an almost complex manifold is a complex cycle if and only if it is pseudoholomorphic.
\end{eg}

\begin{eg}
An oriented manifold $S$, a $k$-current $T$ on $S$, and a proper smooth function \hbox{$f: S \lra X$} induce a $k$-current $f_*T$, called the \textit{pushforward}, on $X$ via the formula
$$f_*T(\omega) := T(f^*\omega).$$
If $T$ is a cycle, so is $f_*T$.
\end{eg}

\begin{eg}
A continuous function $f$ from a triangulated space $S$ to $\R^n$ is \textit{piecewise smooth} if it is smooth on each face of $S$. If $S$ is compact and oriented, then the sum of the pushforwards of the faces of $S$ under $f$ determines a current $f_*S$ on $\R^n$.
\end{eg}

\begin{eg}\label{polyhedral_eg}
Integration over an oriented affine $k$-simplex in $\R^n$ induces a $k$-current on $\R^n$. An element~$P$ of the real vector space generated by such currents is called a \textit{polyhedral chain} (cf \cite[4.1.22 and 4.1.32]{GMT}). By reversing orientations, the coefficients can always be chosen positive. If $\Delta^k \subset \R^k$ is the standard $k$-simplex, every $k$-polyhedral current can be expressed as $P = \sum_i r_i p_{i*}\Delta^k$, where $r_i > 0$, and $p_i$ are injective affine maps $\Delta^k \lra \R^n$. As $p_i$ are affine, the images of $\Delta^k$ under $p_i$ intersect only in affine simplices. By subdividing, we may assume the images intersect only along the boundary.
\end{eg}

\begin{eg}
A differential $(n-k)$-form $\alpha$ induces a $k$-current via the formula
$$\alpha(\omega) = \int_X \alpha \wedge \omega.$$
The boundary of the current induced by $\alpha$ is the current induced by $\d\alpha$, up to sign. A nonzero form on an almost complex manifold is a complex cycle if it is a closed semipositive $(n-1, n-1)$-form.
\end{eg}


\begin{lmm}[{\cite[Theorem IV.18.14]{deRham}}]\label{constancy_lmm}
If $X$ is a connected $n$-manifold and $T$ is an $n$-cycle on $X$, then $T$ is a constant function.
\end{lmm}

A Riemannian metric on $X$ induces a norm $|\cdot|$ on $\Lambda^kT_xX$ for every $x \in X$. This induces a norm~$|\cdot|$ on $\Lambda^k(T^*_xX)$, called the \textit{comass}, via the formula
$$|\omega| := \sup\{ \omega(v) : v \in \Lambda^k(T_xX), v~\textnormal{simple}, |v| \leq 1\}.$$
A subset $U \subset X$ determines a seminorm on $\Omega_c^k(X)$ via the formula
$$||\omega||_U = \sup\bigl\{\big|\omega|_x\big| : x \in U\bigr\}.$$
The \textit{mass} and \textit{flat seminorm with respect to $U$} of a $k$-current $T$ are defined as
\begin{align*}
\mass(T) &:= \sup\{T(\omega) : \omega \in \Omega_c^k(X), ||\omega||_X \leq 1\} \quad \textnormal{and}\\
\flat_U(T) &:= \sup\{T(\omega) : \omega \in \Omega_c^k(X), ||\omega||_U \leq 1, ||\d\omega||_U \leq 1 \},
\end{align*}
respectively.

\begin{eg}
The mass of the current induced by an oriented embedded surface $S$ is the area of~$S$. More generally, the mass of the current induced by an oriented immersed surface $S$ in $X$ with only transverse self-intersections is the area of~$S$.
\end{eg}

The next lemma follows immediately from the relevant definitions.

\begin{lmm}\label{mass_weak_lmm}
If $T_i$ are currents with finite mass converging in the weak topology to a current $T$ with finite mass, then 
$$\mass(T) \leq \liminf\big(\mass(T_i)\big).$$
\end{lmm}

\begin{lmm}[{\cite[Corollary 7.3]{FF}}]\label{flat_implies_weak_lmm}
If $X$ is a compact manifold and $T_i$ is a sequence of currents on $X$ with $\mass(T_i)$ and $\mass(\partial T_i)$ bounded, then $T_i$ converge to a current $T$ in the weak topology on currents if and only if $\flat_X(T_i - T)$ converges to zero.
\end{lmm}

\begin{lmm}[{\cite[4.1.14]{GMT}}]
If $f$ is a Lipschitz function and $T$ is a compactly supported current such that $\mass(T)$ and $\mass(\partial T)$ are both finite, then the pushforward $f_*T$ can be defined. 
\end{lmm}

\begin{lmm}[{\cite[4.1.13]{GMT}}]\label{flat_C1_lmm}
Let $X$ and $S$ be manifolds and $T$ be a compactly supported $k$-current on~$S$ such that $\mass(T)$ and $\mass(\partial T)$ are finite. For $K > 0$, let $\mathfrak{C}_K(S,X)$ be the space of smooth functions from $S$ to $X$ with the norms of their Jacobians bounded by $K$, equipped with the $C^0$-topology. The map
$$p: \mathfrak{C}_K(S,X) \lra \big(\Omega_c^k(X)\big)^*, \quad \quad \bigl\{p(f) \bigl\}(\omega) = T(f^*\omega),$$
is continuous with respect to the flat seminorm $\flat_X$.
\end{lmm}

In particular, $\flat(f_*S)$ is close to $\flat(g_*S)$ for $C^1$-close smooth maps $f$ and $g$. In contrast, it does not follow in general that $\mass(f_*S)$ is close to $\mass(g_*S)$. In order to maintain control on $f_*S$ in this situation, we use the classical notion of the \textit{area of a map}, denoted $\area(f)$, which is 
the integral of the norm of the Jacobian of $f$, denoted $|Df|$, over $S$ (cf.~\cite[Section 3.2.1]{GMT}). The area of an immersion with only transverse self-intersections is the area of the image. The next lemma is established at the end of this section via a standard smoothing argument.

\begin{lmm}\label{smoothing_lmm}
Let $U \subset \R^n$ be an open set. Let $S$ be a closed oriented manifold and
$\gee : S \lra U$
be a piecewise smooth map. For every $\varepsilon > 0$, there exists a smooth map
$f: S \lra U$
such that
\begin{equation}\label{smoothing_bound_eqn}
\flat_U\big( f_*S - \gee_*S \big) < \varepsilon \quad \textnormal{and} \quad 
\area(f) < \area(\gee) + \varepsilon.  
\end{equation}
\end{lmm}


If the Riemannian metric is compatible with the almost complex structure, the mass is particularly well-behaved. A nondegenerate 2-form~$\omega$ on an almost complex manifold $(X,J)$ is \textit{almost Hermitian} if
$$\omega(v,Jv) > 0 \quad \textnormal{and} \quad \omega(Jv,Jw) = \omega(v,w) \quad \forall v,w \in T_xX, ~x \in X, ~v \neq 0.$$
Such a form $\omega$ induces a Riemannian metric $g := \omega(-,J-)$. The manifold $X$ itself is then also called almost Hermitian.

\begin{lmm}[{\bf{Wirtinger's Inequality} \textnormal{\cite[p40]{GMT}}}]\label{wirtinger_lmm}
Let $(X,J,\omega)$ be an almost Hermitian manifold. For any oriented $2$-plane $\Pi$ in~$TX$,
$$\omega|_\Pi \leq \vol_\Pi$$
with equality if and only if $\Pi$ is $J$-invariant and equipped with the orientation induced by $J$.
\end{lmm}



\begin{eg}\label{complex_cycles_are_normal_eg}
If $\omega'$ is another 2-form with $||\omega'||_X \leq 1$, then $(\omega - \omega')(v,Jv) \geq 0$ for all $v \in TX$. Thus, $T(\omega) \geq T(\omega')$ if $T$ is a complex cycle as in Definition~\ref{complex_cycle_dfn}. Therefore, the mass of a complex cycle $T$ with respect to an almost Hermitian metric $\omega$ is $T(\omega)$.
\end{eg}


\begin{dfn}\label{kahler_defect_dfn}
Let $(X,J,\omega)$ be an almost Hermitian manifold and $\Pi$ an oriented 2-plane in $TX$. The \textit{Kähler defect} of $\Pi$ is
$$\KD(\Pi) := 1 - \frac{\omega|_\Pi}{\vol_\Pi}.$$
\end{dfn}

Equivalently, $\KD = 1 - \cos(\theta)$, where $\theta$ is the Kähler angle as defined in \cite{Don}.
For any oriented 2-plane~$\Pi \subset TX$,
$0 \leq \KD(\Pi) \leq 2$.
Furthermore, $\KD(\Pi) = 0$ if and only if $\Pi$ is $J$-invariant and equipped with the orientation induced by $J$.
Therefore, the Kähler defect can serve as the distance function $D$ in Definition~\ref{coarse_curve_dfn}. This is the distance function we use for the rest of the paper.

\begin{proof}[{\bf{\emph{Proof of Lemma \ref{smoothing_lmm}}}}]
Let $\phi: \R^m \lra \R$ be a nonnegative, compactly supported smooth function with total integral $1$. This generates a family of smooth functions, called \textit{smoothing kernels} (cf. \cite[4.1.2]{GMT}), via the formula
$$\phi_\delta(x) := \delta^{-m} \phi (\delta^{-1}x).$$
For a continuous function $g$ on $\R^m$ and each $\delta > 0$, the convolution $g_\delta := g * \phi_{\delta}$ is a smooth function on~$\R^m$. As $\delta$ tends to $0$, $g_\delta$ converges uniformly on compact sets to $g$. If $g$ is smooth, $g_\delta$ converges uniformly with all derivatives on compact sets to $g$ as $\delta$ tends to zero. If $g$ is $K$-Lipschitz, a routine calculation shows that $g_\delta$ is $K$-Lipschitz for every $\delta > 0$.\\

Embed $S$ smoothly in $\R^m$. Extend $\gee$ to a compactly supported piecewise smooth map defined on all of $\R^m$. The function $\gee$ is Lipschitz, so $\gee_\delta$ is $K$-Lipschitz for every $\delta > 0$ and some $K$ not dependent on $\delta$. For every $\delta > 0$ sufficiently small, $\gee_\delta$ maps $S$ into $U \subset \R^n$. We now show that $f := \gee_\delta$ satisfies (\ref{smoothing_bound_eqn}) for every $\delta > 0$ sufficiently small. It suffices to check this for each face~$\Delta$ of $S$. By Lemma \ref{flat_C1_lmm}, $f_*\Delta$ is $\flat_U$-close to $\gee_*\Delta$ for every $\delta > 0$ sufficiently small. 
Take a compact~$A \subset \Delta$ such that $\Delta - A$ has area less than $\frac{\varepsilon}{4K^2}$. For every $\delta > 0$ sufficiently small, $\bigl|Df|_A\bigr|$ is within $\frac{\varepsilon}{2\area(A)}$ of $\bigl|D\gee|_A\bigr|$. Also, as $f$ and $g$ are $K$-Lipschitz, $|Df| \leq K^2$ and $|Dg| \leq K^2$. Therefore,
\begin{align*}
\bigl|\area(\gee|_\Delta) - \area(f|_\Delta)\bigr| & = \biggl|\int_\Delta \big(|D\gee| - |Df|\big)\vol\biggr| \\
& \leq \biggl|\int_A\big(|D\gee| - |Df|\big)\vol \biggr| + \biggl| \int_{\Delta -A}\big(|D\gee| - |Df|\big)\vol\biggr| < \varepsilon.
\end{align*}
Therefore, $\delta$ can be chosen so that $f$ satisfies (\ref{smoothing_bound_eqn}).
\end{proof}

\section{Proofs of the Main Results}\label{Main_Results_sec}

In this section, we prove the results stated in Section \ref{Introduction_sec}, except Theorem~\ref{lower_energy_bound_thm}. By Definition~\ref{kahler_defect_dfn}, a surface $S$ is a coarsely holomorphic curve if and only if its surface area is close to the integral of an almost Hermitian form over $S$. That is, coarse holomorphicity is determined entirely by integration of forms and is thus a measure-theoretic condition. Therefore, with some measure-theoretic control on~$S$, the main results of this paper follow quickly. This control is provided by Proposition \ref{key_prp}, which is proved in Section \ref{Prop_Proof_sec}.

\begin{proof}[{\bf{\emph{Proof of Proposition \ref{Rel_cycle_prp}}}}]
The space $\Omega^+(X)$
forms a convex cone in $\Omega_{c}^{2}(X)$. It is open in the strong $C^\infty$-topology. The space $\widetilde{V}$ of all closed 2-forms whose cohomology classes lie in $V$ is a linear subspace. By assumption, $\widetilde{V} \cap \Omega^+(X) = \emptyset$. By the Hahn-Banach Theorem \cite[Theorem II.3.1]{TVS}, there is a closed hyperplane containing $\widetilde{V}$ and disjoint from $\Omega^+(X)$. Taking the quotient by this hyperplane yields a continuous linear functional $T$ on $\Omega_{c}^{2}(X)$ which is positive on $\Omega^+(X)$ and vanishes on $\widetilde{V}$. The functional $T$ is, by definition, a complex cycle such that $\langle h, [T] \rangle = 0$ for all $h \in V$.
\end{proof}

\begin{prp}\label{key_prp}
Let $(X,g)$ be a compact Riemannian manifold of dimension at least 4. For every 2-cycle $T$ on $X$ with finite mass,
there exist sequences $c_i > 0$ of numbers and $S_i$ of closed oriented surfaces embedded in X such that $c_i S_i$ tends to $T$ in the weak topology of currents and
$\mass(c_i S_i)$ tends to~$\mass(T)$.
\end{prp}

\begin{proof}[{\bf{\emph{Proof of Theorem \ref{Approx_thm}}}}]
This follows immediately from Lemma \ref{constancy_lmm} when the dimension of $X$ is 2, so we assume the dimension of $X$ is at least 4. By Example \ref{complex_cycles_are_normal_eg}, the cycle $T$ has finite mass. Therefore, by Proposition \ref{key_prp}, there exist sequences $c_i > 0$ and of closed oriented surfaces $S_i$ embedded in $X$ such that $c_i S_i$ tends to $T$ in the weak topology of currents and
$\mass(c_i S_i)$ tends to~$\mass(T)$. Therefore,

$$\lim_{i\to\infty}c_i \int_{S_i} \omega = T(\omega) = \mass(T) = \lim_{i\to\infty} \mass(c_i S_i) = \lim_{i\to\infty} c_i \int_{S_i} \vol \,.$$
So,
$$\lim_{i\to\infty} c_i \int_{S_i} (\vol - \omega) = 0.$$
As $\lim_{i\to\infty}c_i \,\area (S_i) = T(\omega) \neq 0$,
$$0 = \lim_{i\to\infty}\frac{c_i \int_{S_i} (\vol - \omega)}{c_i \int_{S_i} \vol} = \lim_{i\to\infty} \frac{\int_{S_i} \KD\, \vol}{\int_{S_i} \vol} \,.$$
Therefore, $S_i$ is coarsely $\varepsilon_i$-holomorphic for a sequence $\varepsilon_i > 0$ tending to zero.
\end{proof}

\begin{proof}[{\bf{\emph{Proof of Corollary \ref{Existence_crl}}}}]
By Proposition \ref{Rel_cycle_prp} with $V=0$, every closed almost complex manifold contains a complex cycle $T \neq 0$. By Theorem \ref{Approx_thm}, $T$ is approximated by coarsely $\varepsilon_i$-holomorphic curves for a sequence $\varepsilon_i > 0$ tending to zero.
\end{proof}

\begin{proof}[{\bf{\emph{Proof of Theorem \ref{Gromov_conj_thm}}}}]
The ``only if'' direction is clear. Therefore, assume that the cohomology class $h \in H^2(X, \R)$ contains no symplectic form taming $J$, but satisfies (\ref{area_bound_eqn}). Then, by Corollary~\ref{Existence_crl}, $-h$ does not contain a taming symplectic form, so the linear subspace in $H^2(X,\R)$ generated by $h$ does not contain the class of a taming symplectic form. By Proposition \ref{Rel_cycle_prp}, there exists a complex cycle $T \neq 0$ on $X$ such that $\langle h, [T] \rangle = 0$. Therefore, by Theorem \ref{Approx_thm}, there exist sequences $\varepsilon_i >0$, $c_i > 0$, and of $S_i$ such that $S_i$ is a coarsely $\varepsilon_i$-holomorphic curve, $\varepsilon_i$ tends to zero,
$$\lim_{i\to\infty} c_i \langle h, [S_i] \rangle = 0, \quad \textnormal{and} \quad \lim_{i\to\infty} c_i~\area(S_i) \neq 0.$$
Therefore,
$$\lim_{i\to\infty}\frac{\langle h, [S_i] \rangle}{\area(S_i)} = 0,$$
contradicting the assumption that $h$ satisfies (\ref{area_bound_eqn}).
\end{proof}

\begin{proof}[{\bf{\emph{Proof of Corollary \ref{symp_cone_crl}}}}]
The ``only if'' direction is clear. Therefore, suppose there exists $\varepsilon > 0$ such that the cohomology class $h \in H^2(X, \R)$ is strictly positive on coarsely $\varepsilon$-holomorphic curves. Take a symplectic form $\omega$ on $X$ which tames $J$. By the finite dimensionality of $H_2(X,\R)$, there exists $C > 0$ such that $\langle [\omega], [S] \rangle \leq C\langle h, [S] \rangle$ for all coarsely $\varepsilon$-holomorphic curves $S$. Thus,
$$\area(S) \leq C'\langle [\omega], [S] \rangle \leq C'C \langle h, [S] \rangle,$$
where $C'$ is the constant from Theorem~\ref{Gromov_conj_thm} applied to $\omega$. Therefore, $h$ contains a symplectic form taming $J$ by Theorem \ref{Gromov_conj_thm}.
\end{proof}

\section{Approximation of Cycles}\label{Prop_Proof_sec}

In this section, we prove Proposition~\ref{key_prp}. The argument proceeds by approximating a 2-cycle with finite mass by a sequence of nicer geometric objects. It consists mostly of classical arguments, interspersed with concrete constructions. In particular, Lemma~\ref{polyhedral_lmm} follows from the Federer-Fleming Deformation Theorem and Lemma~\ref{parameterize_lmm} is a simple desingularization procedure.
These results and Lemma \ref{smoothing_lmm} yield Corollary~\ref{smooth_approx_crl}. Lemma~\ref{surgery_lmm} is an explicit construction of a surgery. The proof of Proposition~\ref{key_prp} uses the $C^1$~Nash Embedding Theorem to apply Corollary~\ref{smooth_approx_crl} and concludes with the Whitney Immersion Theorem and Lemma~\ref{surgery_lmm}. The only part of the argument that doesn't generalize immediately to the case of an $n$-cycle with finite mass in a manifold of dimension at least $2n$ is Lemma~\ref{parameterize_lmm}. This result could likely also be generalized, for instance using \cite[Théorème II.4]{Thom}, but this would take us too far afield from our current goals.

\begin{lmm}\label{polyhedral_lmm}
Let $U \subset \R^n$ be a precompact open set. Let $T$ be a 2-cycle on $\R^n$ with finite mass and support contained in $U$. For every $\varepsilon > 0$, there exists a polyhedral 2-cycle $P$ on $\R^n$ with rational coefficients and support contained in $U$ such that
$$\flat_U(P - T) < \varepsilon \quad \textnormal{and} \quad \mass(P) < \mass(T) + \varepsilon.$$
\end{lmm}

\begin{proof}
By \cite[4.2.24]{GMT}, there is a sequence $P'_i$ of polyhedral $2$-chains supported inside $U$ such that $P'_i$ tends to $T$ in the weak topology and $$\limsup\big(\mass(P'_i) + \mass(\partial P'_i)\big) \leq \mass(T).$$
By Lemma \ref{mass_weak_lmm}, $\mass(T) \leq \liminf \big( \mass(P'_i)\big)$. It follows that $\mass(P'_i)$ tends to $\mass(T)$ and thus that $\mass(\partial P'_i)$ tends to $0$.
Therefore, by \cite[Theorem 7.1]{FF}, there exist currents $G_i$ and $K_i$ supported inside $U$ such that
$$ G_i + \partial K_i = P'_i - T$$ and $\mass(G_i)$ tends to $0.$
Taking the boundary of each side shows $\partial G_i = \partial P'_i$ and so $\partial G_i$ is polyhedral. By the Federer-Fleming Deformation Theorem \cite[4.2.9]{GMT},
$$G_i = P_i + Q_i + \partial S_i,$$ where $P_i$, $Q_i$, and $\partial S_i$ are supported in $U$, $P_i$ and $Q_i$ are polyhedral chains, $\mass(P_i)$ can be taken to be less than a constant multiple of $\mass(G_i)$, and $\mass(Q_i)$ can be taken arbitrarily small. Therefore, $\mass(P_i + Q_i)$ tends to $0$. Again, taking the boundary of each side shows $\partial G_i = \partial (P_i + Q_i)$. Therefore, $P'_i - (P_i + Q_i)$ are polyhedral cycles supported inside $U$ converging in the weak topology to $T$ such that $\mass\big(P'_i - (P_i + Q_i) \big)$ tends to $\mass(T)$. By Lemma \ref{flat_implies_weak_lmm}, $P'_i - (P_i + Q_i)$ tends to $T$ in the seminorm $\flat_U$.\\

The condition on the coefficients for a polyhedral chain to be a cycle is a linear system over the integers. Therefore, the rational solutions are dense in the real solutions and $P'_i - (P_i + Q_i)$ can be chosen to have rational coefficients. Taking $i$ large enough yields a polyhedral $2$-cycle $P$ with the stated properties.
\end{proof}
\begin{lmm}\label{parameterize_lmm}
Let $U \subset \R^n$ be an open set. Let $P$ be a polyhedral 2-cycle on $\R^n$ with integer coefficients and support contained in $U$. Then, there exist a closed oriented surface $S$ and a piecewise linear map
$\gee : S \lra U$
such that
$$\gee_*S = P \quad \textnormal{and} \quad \area(\gee) = \mass(P).$$
\end{lmm}

\begin{proof}
As in Example \ref{polyhedral_eg}, let $P$ be represented as $\sum r_ip_{i*}\Delta^2$, where $r_i$ are positive integers, $p_i$ are injective affine maps $\Delta^2 \lra \R^n$, and the images of $\Delta^2$ under $p_i$ intersect only along the boundary. Let $\mathcal{A}_k$ denote the set with $k := \sum r_i$ elements and let the map
$$p: \mathcal{A}_k \times \Delta^2 \lra \R^n$$
act as $p_1$ on the first $r_1$ copies of $\Delta^2$, as $p_2$ on the next $r_2$ copies of $\Delta^2$, and so on.
It is clear by construction that
$\area(p') = \mass(P)$.\\

\tikzset{every picture/.style={line width=0.75pt}} 
\begin{figure}
\centering
\begin{tikzpicture}[x=0.75pt,y=0.75pt,yscale=-1,xscale=1]

\draw   (300,70) -- (300,110) -- (260,50) -- cycle ;
\draw   (330,70) -- (330,110) -- (370,50) -- (336.78,66.61) -- cycle ;
\draw   (240,80) -- (280,100) -- (240,40) -- cycle ;
\draw   (390,80) -- (350,100) -- (390,40) -- cycle ;
\draw   (310,140) -- (310,180) -- (270,160) -- cycle ;
\draw   (320,140) -- (320,180) -- (360,160) -- cycle ;
\draw   (310,130) -- (270,110) -- (270,150) -- cycle ;
\draw   (360,110) -- (360,150) -- (320,130) -- cycle ;
\draw   (460,190) -- (460,230) -- (420,170) -- cycle ;
\draw   (460,190) -- (460,230) -- (500,170) -- cycle ;
\draw   (420,210) -- (460,230) -- (420,170) -- cycle ;
\draw   (500,210) -- (460,230) -- (483.73,194.41) -- (500,170) -- cycle ;
\draw  [dash pattern={on 0.84pt off 2.51pt}] (500,170) -- (500,210) -- (460,190) -- cycle ;
\draw  [dash pattern={on 0.84pt off 2.51pt}] (460,190) -- (460,230) -- (500,210) -- cycle ;
\draw  [dash pattern={on 0.84pt off 2.51pt}] (460,190) -- (420,170) -- (420,210) -- cycle ;
\draw  [dash pattern={on 0.84pt off 2.51pt}] (460,190) -- (460,230) -- (420,210) -- cycle ;
\draw   (160,190) -- (160,230) -- (120,170) -- cycle ;
\draw   (170,190) -- (170,230) -- (210,170) -- cycle ;
\draw   (120,210) -- (160,230) -- (120,170) -- cycle ;
\draw   (210,210) -- (170,230) -- (193.73,194.41) -- (210,170) -- cycle ;
\draw  [dash pattern={on 0.84pt off 2.51pt}] (210,170) -- (210,210) -- (170,190) -- cycle ;
\draw  [dash pattern={on 0.84pt off 2.51pt}] (170,190) -- (170,230) -- (210,210) -- cycle ;
\draw  [dash pattern={on 0.84pt off 2.51pt}] (160,190) -- (120,170) -- (120,210) -- cycle ;
\draw  [dash pattern={on 0.84pt off 2.51pt}] (160,190) -- (160,230) -- (120,210) -- cycle ;
\draw    (250,200) -- (378,200) node[midway, below] {$g$};
\draw [shift={(380,200)}, rotate = 180] [color={rgb, 255:red, 0; green, 0; blue, 0 }  ][line width=0.75]    (10.93,-3.29) .. controls (6.95,-1.4) and (3.31,-0.3) .. (0,0) .. controls (3.31,0.3) and (6.95,1.4) .. (10.93,3.29)   ;
\draw    (240,130) -- (171.84,159.21) node[midway, above] {$\pi$};
\draw [shift={(170,160)}, rotate = 336.8] [color={rgb, 255:red, 0; green, 0; blue, 0 }  ][line width=0.75]    (10.93,-3.29) .. controls (6.95,-1.4) and (3.31,-0.3) .. (0,0) .. controls (3.31,0.3) and (6.95,1.4) .. (10.93,3.29)   ;
\draw    (390,130) -- (458.16,159.21) node[midway,above] {$p$};
\draw [shift={(460,160)}, rotate = 203.2] [color={rgb, 255:red, 0; green, 0; blue, 0 }  ][line width=0.75]    (10.93,-3.29) .. controls (6.95,-1.4) and (3.31,-0.3) .. (0,0) .. controls (3.31,0.3) and (6.95,1.4) .. (10.93,3.29)   ;

\node at (100,200) {$S$};
\node at (520,200) {$P$};
\node at (316,50) {$\mathcal{A}_k \times \Delta^2$};

\end{tikzpicture}
\hspace{3mm}
\caption{Construction of a smooth surface parameterizing an integer polyhedral cycle}
\end{figure}
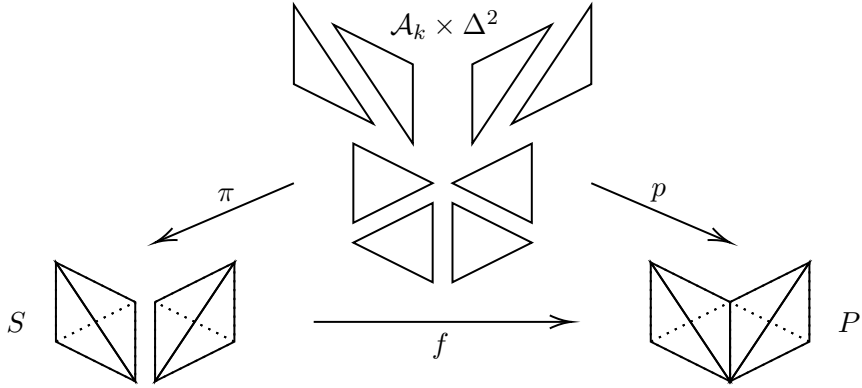

We glue the edges of $\mathcal{A}_k \times \Delta^2$ together to construct a closed, oriented surface $S$ such that $p$ factors through the quotient map $\pi: \mathcal{A}_k \times \Delta^2 \lra S$ as follows:
\[
 \begin{tikzcd}[column sep=small]
& \mathcal{A}_k \times \Delta^2 \arrow{dl}[swap]{\pi} \arrow[dr, "p"] & \\
  S \arrow{rr}[swap]{\gee} &                         & U
\end{tikzcd}
\]
As $P$ is a cycle, for every $i \leq k$ and every edge $e$ of $\{ i\} \times \Delta^2$, there is a $j \leq k$ and edge $e'$ of $\{ j \}\times \Delta^2$ such that $p_*e = -p_*e'$. Therefore, $p_j^{-1} \circ p_i$ is an orientation-reversing affine map from $e$ to $e'$. We associate the edges of $\mathcal{A}_k \times \Delta^2$ in pairs along such maps to construct $S$. All that remains to show is that $S$ is a manifold, as all of its other claimed properties are clear by construction. The surface $S$ is clearly locally Euclidean in the interior of each face. As every edge in $S$ lies in two faces, $S$ is locally Euclidean in the interior of each edge. Lastly, by construction, the link of each vertex is connected, and is therefore a circle. Thus, $S$ is locally Euclidean at each vertex.
\end{proof}

\begin{crl}\label{smooth_approx_crl}
Let $U \subset \R^n$ be a precompact open set. Let $T$ be a 2-cycle on $\R^n$ with finite mass and support contained in $U$. For every $\varepsilon > 0$, there exist a closed, oriented surface $S$, a smooth map
$f : S \lra U,$
and a positive constant $c$ such that
$$\flat_U\big(c \cdot f_*S - T\big) < \varepsilon \quad \textnormal{and} \quad c \cdot \area(f) < \mass(T) + \varepsilon.$$
\end{crl}

\begin{proof}
By Lemma \ref{polyhedral_lmm}, there exists a polyhedral 2-cycle with rational coefficients $P$ with support contained in $U$ such that
$$\flat_U(P - T) < \frac{\varepsilon}{2} \quad \textnormal{and} \quad \mass(P) < \mass(T) + \frac{\varepsilon}{2}.$$
As the coefficients of $P$ are rational, $P = cP'$ for $c$ the reciprocal of some positive integer and $P'$ a polyhedral 2-cycle with integer coefficients. By Lemma \ref{parameterize_lmm}, there exist a closed oriented surface~$S$ and a piecewise linear map
$\gee : S \lra U$
such that
$$\gee_*S = P' \quad \textnormal{and} \quad \area(\gee) = \mass(P').$$
By Lemma \ref{smoothing_lmm}, there exists a smooth map
$f: S \lra U$
such that
$$\flat_U(f_*S - \gee_*S) < \frac{\varepsilon}{2c} \quad \textnormal{and} \quad \area(f) < \area(\gee) + \frac{\varepsilon}{2c}.$$
Therefore,
$$\flat_U\big(c \cdot f_*S - T\big) \leq \flat_U\big(c \cdot f_*S - c \cdot \gee_*S\big) + \flat_U\big(c \cdot \gee_*S - P\big) + \flat_U(P - T) < \varepsilon,$$
$$c \cdot \area(f) < c \cdot \area(\gee) + \frac{\varepsilon}{2} = \mass(P) + \frac{\varepsilon}{2} < \mass(T) + \varepsilon,$$
as needed.
\end{proof}

\begin{lmm}\label{surgery_lmm}
Let $(X,g)$ be a Riemannian manifold of dimension at least 4. For every immersed oriented surface $S \subset X$ with only transverse self-intersections and every $\varepsilon > 0$, there exists
an embedded oriented surface $\widetilde{S} \subset X$ such that
\begin{equation}\label{surgery_eqn}
\flat_X\big(\widetilde{S} - S\big) < \varepsilon \quad \textnormal{and} \quad \area\big(\widetilde{S}\,\big) < \area(S) + \varepsilon.
\end{equation}
\end{lmm}

\begin{proof}
This is automatic if the dimension of $X$ is greater than 4, so we assume that it is 4. We remove the self-intersections of $S$ via a local surgery. Take a coordinate chart around a self-intersection point of $S$ such that, in this chart, $S$ is the union of the $(x_1,x_2)$-plane and the $(x_3,x_4)$-plane in $\R^4$.
Define $h: S^1 \times [0,1] \lra \R^4$ by
$$h(e^{i\theta}, t) = (e^{-1/t}\cos\theta, e^{-1/t}\sin\theta, e^{1/(t-1)}\cos\theta, e^{1/(t-1)}\sin\theta).$$
Let $B$ be the ball of radius $e^{-1}$ in the Euclidean coordinates of the chart. As all the derivatives of~$e^{-1/t}$ tend to zero as $t$ tends to zero, the subset
$$\widetilde{S} := h\big( S^1 \times [0,1] \big) \cup (S - B)$$
is a smooth, embedded submanifold which coincides with $S$ outside of $B$. We can extend the orientation from $S-B$ to all of $\widetilde{S}$, possibly after postcomposing $h$ with a reflection across a coordinate hyperplane. Therefore, by taking $r > 0$ sufficiently small and replacing the part of $S$ in this coordinate chart with $r\widetilde{S}$, one constructs an immersed oriented surface with only transverse self intersections with one fewer self-intersection satisfying~(\ref{surgery_eqn}). The result follows by induction on the number of self-intersections.
\end{proof}

\begin{proof}[{\bf{\emph{Proof of Proposition \ref{key_prp}}}}]
By the $C^1$ Nash Embedding Theorem \cite{Nash}, we can assume $X$ is a $C^1$-submanifold of $\R^n$ for some $n$ and that $g$ is the metric induced by the Euclidean metric on~$\R^n$. For $\varepsilon > 0$,
there exists a tubular neighborhood $U \supset X$ and a $(1 + \varepsilon)$-Lipschitz retraction $\rho: U \lra X$. Take $S$ and $f$ as in Corollary \ref{smooth_approx_crl} so that
$$f(S) \subset U, \quad \flat_U\big(c \cdot f_*S- T\big) < \varepsilon, \quad \textnormal{and} \quad c \cdot \area(f) < \mass(T) + \varepsilon.$$
Then, $\rho \circ f: S \lra X$ is defined and satisfies
$$\flat_X\big(c \cdot (\rho \circ f)_*S - T\big) < (1 + \varepsilon)^2\varepsilon \quad \textnormal{and} \quad c \cdot \area(\rho \circ f) < (1 + \varepsilon)^2(\mass(T) + \varepsilon).$$
By the Whitney Immersion Theorem \cite[Theorem 3.2.9]{Hirsch}, there is a smooth immersion $h: S \lra X$ with only transverse self-intersections which is $C^1$-close to $\rho \circ f$. As the area of a map is continuous in the $C^1$-topology and by Lemma \ref{flat_C1_lmm}, $h$ can be chosen so
$$c \cdot \flat_X\big(h_*S - (\rho \circ f)_*S\big) < \varepsilon \quad \textnormal{and} \quad c \cdot \area(h) < c \cdot \area(\rho \circ f) + \varepsilon.$$
By Lemma \ref{surgery_lmm}, there is an embedded closed oriented surface $\widetilde{S} \subset X$ such that
$$c \cdot \flat_X\big( \widetilde{S} - h_*S \big) < \varepsilon \quad \textnormal{and} \quad c \cdot \area\big(\widetilde{S}\,\big) < c \cdot \area(h) + \varepsilon.$$
Therefore,
\begin{align*}
    c \cdot \flat_X\big(\widetilde{S} - T\big) &\leq c \cdot \flat_X\big(\widetilde{S} - h_*S  \big) + c \cdot \flat_X\big(h_*(S) - (\rho \circ f)_*S  \big) + \flat_X\big(c \cdot (\rho \circ f)_*S  - T  \big)\\
    &< 2\varepsilon + (1 + \varepsilon)^2\varepsilon
\end{align*}
and
\begin{align*}
    c \cdot \area(\widetilde{S}) &< c \cdot \area(h) + \varepsilon < c \cdot \area(\rho \circ f) + 2\varepsilon < (1 + \varepsilon)^2(\mass(T) + \varepsilon) + 2\varepsilon\\
    &< \mass(T) + \big(3 + 2\mass(T) + 2\varepsilon + \varepsilon\mass(T) + \varepsilon^2 \big)\varepsilon.
\end{align*}
The result now follows from Lemma \ref{flat_implies_weak_lmm}.
\end{proof}

\section{Lower Energy Bound}\label{lower_energy_section}

In this section, we prove Theorem \ref{lower_energy_bound_thm}. The proof itself is rather simple. It uses the Federer-Fleming Compactness Theorem (Proposition \ref{FF_compactness_prp}) and Lemma \ref{semicalibration_lmm} to reduce to the classical case of pseudoholomorphic curves. However, the background for the proof, especially Lemma \ref{semicalibration_lmm}, is quite involved. It depends on the regularity theory of semicalibrated currents, as developed in \cite{RT} in the locally symplectic case and in \cite{DSS1,DSS2,DSS3,DSS4} in general. We review the necessary background, following \cite[Sections 4 and 5]{DW} closely. We use $\mathscr{H}^k$ to denote the $k$-dimensional Hausdorff measure of a Riemannian manifold.

\begin{dfn}
Let $(X,g)$ be a Riemannian manifold. A subset $A \subset X$ is \textit{$k$-rectifiable} if there exists a countable set of $C^1$-embedded $k$-submanifolds $A_1, A_2, \dotsc \subset X$ such that
$$A \subset \bigcup_1^\infty A_i \quad \textnormal{and} \quad \mathscr{H}^k\Big(\bigcup_1^\infty A_i - A\Big) = 0.$$
\end{dfn}

If $A \subset X$ is $k$-rectifiable, then for $\mathscr{H}^k$-almost every $x \in X$, there exists an \textit{approximate tangent space to $A$ at $x$}. It is a $k$-dimensional subspace of $T_x X$ and is denoted $\pi(A,x)$. For more details, see \cite[3.2.16]{GMT}.

\begin{dfn}
Let $(X,g)$ be a Riemannian manifold. A $k$-current $T$ on $X$ is \textit{integer rectifiable} if $\mass(T)$ is finite and there exist
\begin{enumerate}
\item a $k$-rectifiable subset $A \subset X$,
\item an $\mathscr{H}^k$-measurable function $\mathfrak{m}: A \lra \N \cup 0$, and
\item an $\mathscr{H}^k$-measurable section $\overrightarrow{T}$ of $\Lambda^k TM|_A$
\end{enumerate}
such that
\begin{enumerate}
\item for $\mathscr{H}^k$-almost all $x \in A$, the $k$-vector $\overrightarrow{T}(x)$ is given by $\overrightarrow{T}(x) = e_1 \wedge \dots \wedge e_k$ for an orthonormal frame $e_i$ of $\pi(A,x)$, and
\item $T$ can be expressed as
$$T(\omega) = \int_A \mathfrak{m}(x) \omega(\overrightarrow{T}(x)) \d \mathscr{H}^k,$$
for $\omega \in \Omega^k_c(X)$.
\end{enumerate}
A current $T$ is \textit{integral} if both $T$ and $\partial T$ are integer rectifiable.
\end{dfn}

\begin{eg}
Closed submanifolds and pseudoholomorphic curves are integral currents.
\end{eg}

\begin{prp}[{\bf{Federer-Fleming Compactness Theorem} \textnormal{\cite[4.2.17]{GMT}}}]\label{FF_compactness_prp}
Let $(X,g)$ be a compact Riemannian manifold. For every $c > 0$, the space of integral $k$-currents $T$ on $X$ such that $\mass(T)+\mass(\partial T) \leq c$ is compact in the weak topology.
\end{prp}

Proposition \ref{FF_compactness_prp} is stated in \cite{GMT} for compactly supported currents on $\R^n$. The Nash Embedding Theorem~\cite{Nash} immediately implies that it also holds in the form above.

\begin{dfn}
Let $(X,g)$ be a Riemannian manifold. A \textit{semicalibration} is a $k$-form $\omega$ such that $||\omega|| \leq 1$. An integral $k$-current is \textit{semicalibrated} by $\omega$ if $\omega(\overrightarrow{T}(x)) = 1$ for $\mathscr{H}^k$-almost every $x \in \supp(T)$.
\end{dfn}

\begin{eg}\label{complex_cycles_semicalibrated_eg}
Let $(X,J,\omega)$ be an almost Hermitian manifold. By Wirtinger's Inequality (Lemma~\ref{wirtinger_lmm}), $\omega$ is a semicalibration. By Example~\ref{complex_cycles_are_normal_eg}, integral complex cycles are semicalibrated by $\omega$.
\end{eg}

\begin{dfn}\label{J_cycle_dfn}
A \textit{$J$-holomorphic cycle} $S$ is a linear combination
$$S := \sum_1^l n_i S_i\,,$$
where $l >0$, $S_i$ are currents of integration over $J$-holomorphic curves, and $n_i \in \N \cup 0$.
\end{dfn}

\begin{lmm}[{\cite[Lemma 5.5]{DW}}]\label{semicalibration_lmm}
Let $(X,J,\omega)$ be an almost Hermitian manifold. If $T$ is a closed integer rectifiable current with compact support which is semicalibrated by $\omega$, then it is a $J$-holomorphic cycle.
\end{lmm}

\begin{crl}\label{coarse_integral_crl}
Let $(X,J,\omega)$ be a compact almost Hermitian manifold. Suppose there exist sequences $\varepsilon_i > 0$, $n_i \in \N$, and $S_i$, such that $S_i$ is coarsely $\varepsilon_i$-holomorphic, $\varepsilon_i$ tends to zero, and $n_i \, \area(S_i)$ is bounded. Then, there exists a subsequence of $n_i S_i$ converging weakly to a $J$-holomorphic cycle.
\end{crl}

\begin{proof}
By Proposition \ref{FF_compactness_prp}, there exists a subsequence $n_j S_j$ of $n_i S_i$ converging weakly to an integral current $S$.
It follows from $\varepsilon_j$ tending to zero, the definition of weak limit, Lemma~\ref{wirtinger_lmm}, and Lemma~\ref{mass_weak_lmm}, respectively, that
$$\lim \mass(n_j S_j) = \lim n_j \int_{S_j} \omega = S(\omega) \leq \mass(S) \leq \liminf \mass(n_j S_j).$$
Thus, $\mass(S) = S(\omega)$, so $S$ is semicalibrated by $\omega$. Therefore, by Lemma \ref{semicalibration_lmm}, $S$ is a $J$-holomorphic cycle.
\end{proof}

\begin{prp}[\bf{Lower Energy Bound}]\label{classic_energy_prp}
Let $(X,J,\omega)$ be a compact almost Hermitian manifold. There exists $\hbar > 0$ such that 
$$\area(S) > \hbar$$
for every pseudoholomorphic curve $S$.
\end{prp}

Proposition \ref{classic_energy_prp} is well-known (cf. \cite[Proposition 4.2]{Zinger}). It is an immediate consequence of the well-known the Monotonicity Lemma \cite[Proposition 3.12]{Zinger}. 
It is likely that a Monotonicity Lemma for semicalibrated currents like \cite[Proposition~2.1]{DSS4} could prove Theorem~\ref{lower_energy_bound_thm} directly, without appealing to Lemma~\ref{semicalibration_lmm}.

\begin{proof}[{\bf{\emph{Proof of Theorem \ref{lower_energy_bound_thm}}}}]
Let $\hbar$ be as in Proposition \ref{classic_energy_prp}. Suppose (\ref{lower_bound_eqn}) does not hold. Then, there exist sequences $\varepsilon_i > 0$ and $S_i$ such that $S_i$ is coarsely $\varepsilon_i$-holomorphic, $\area(S_i) < \hbar$, and $\varepsilon_i$ tends to zero. Furthermore, there exists a sequence $n_i \in \N$ such that, for $i$ sufficiently large, 
$$\frac{\hbar}{2} \leq n_i \int_{S_i} \omega < \hbar.$$ By Corollary \ref{coarse_integral_crl}, there exists a subsequence $n_j S_j$ of $n_i S_i$ that converges to a $J$-holomorphic cycle~$S$. The cycle is nonzero, because $$\mass(S) = S(\omega) = \lim n_j \int_{S_j} \omega \geq \frac{\hbar}{2}.$$ If $S'$ is a nontrivial summand of $S$ as in Definition \ref{J_cycle_dfn}, then $S'$ is a pseudoholomorphic curve, but $$\area(S') \leq \mass(S) = S(\omega) \leq \hbar,$$
which contradicts the assumption on $\hbar$.
\end{proof}

\vspace{.3in}

{\it Department of Mathematics, Stony Brook University, Stony Brook, NY 11794\\
spencer.cattalani@stonybrook.edu}

\newpage
Data Availability Statement: No data is contained in this manuscript. Conflict of Interest Statement: On behalf of all authors, the corresponding author states that there is no conflict of interest.

\end{document}